\documentclass[10pt]{amsart}

\usepackage{amssymb, amsmath, amsthm, amsfonts}

\usepackage{mathrsfs}
\usepackage{extarrows}
\usepackage{amsthm}
\usepackage{amssymb}

\usepackage[all]{xy}

\usepackage{verbatim}
\usepackage{hyperref}
\hypersetup{
    colorlinks=true,       
    linkcolor=blue,          
    citecolor=blue,        
    filecolor=blue,      
    urlcolor=blue           
}

\newtheorem{theorem}{Theorem}[section]
\newtheorem{cor}[theorem]{Corollary}
\newtheorem{lem}[theorem]{Lemma}
\newtheorem{prop}[theorem]{Proposition}

\theoremstyle{definition}
\newtheorem{defn}[theorem]{Definition}

\theoremstyle{remark}
\newtheorem{rem}[theorem]{Remark}

\newcommand{\HH}{\mathrm{H}}
\newcommand{\End}{\mathrm{End}}
\newcommand{\Hom}{\mathrm{Hom}}

\newcommand{\bA}{\mathbb{A}}

\newcommand{\bP}{\mathbb{P}}

\newcommand{\bQ}{\mathbb{Q}}

\newcommand{\bC}{\mathbb{C}}

\newcommand{\cA}{\mathcal{A}}
\newcommand{\cO}{\mathcal{O}}

\newcommand{\cE}{\mathcal{E}}
\newcommand{\cF}{\mathcal{F}}
\newcommand{\cG}{\mathcal{G}}
\newcommand{\cH}{\mathcal{H}}
\newcommand{\cK}{\mathcal{K}}
\newcommand{\cL}{\mathcal{L}}
\newcommand{\cM}{\mathcal{M}}
\newcommand{\cN}{\mathcal{N}}

\newcommand{\cX}{\mathcal{X}}

\newcommand{\arrow}{\rightarrow}
\newcommand{\rk}{\mathrm{rk}\,}

\newcommand{\Proj}{\mathrm{Proj}}
\newcommand{\Sym}{\mathrm{Sym}}

\newcommand{\HIG}{\mathrm{HIG}}

\newcommand{\rH}{\mathrm{H}}

\title[]{An algebraic approach to the hyperbolicity of moduli stacks of Calabi-Yau varieties}

\author{Yohan Brunebarbe}
\address{Institut de Math\'ematiques de Bordeaux, Universit\'e de Bordeaux, 351 cours de la Lib\'eration, F-33405 Talence}
\email{yohan.brunebarbe@math.u-bordeaux.fr}
\date{\today}

\begin{document}
\maketitle
\begin{abstract}
The moduli stacks of Calabi-Yau varieties are known to enjoy several hyperbolicity properties. The best results have so far been proven using sophisticated analytic tools such as complex Hodge theory. Although the situation is very different in positive characteristic (e.g. the moduli stack of principally polarized abelian varieties of dimension at least 2 contains rational curves), we explain in this note how one can prove many hyperbolicity results by reduction to positive characteristic, relying ultimately on the nonabelian Hodge theory in positive characteristic developed by Ogus and Vologodsky. 
\end{abstract}

\section{Introduction}

\subsection{Main result}

The goal of this note is to provide an algebraic proof of the following result.

\begin{theorem}\label{main result}
Let $S$ be a smooth proper connected complex variety. Let $f: X \arrow S$ be a smooth projective morphism with connected fibres. Assume that the relative canonical bundle $\omega_{X/S}$ is relatively trivial and that the Hodge line bundle $f_\ast(\omega_{X/S})$ is big. Then:
\begin{enumerate}
\item the canonical bundle of $S$ is big (in other words $S$ is of general type),
\item the cotangent bundle of $S$ is big in the sense of Hartshorne,
\item the cotangent bundle of $S$ is weakly positive.
\end{enumerate} 
\end{theorem}

The definitions of bigness and weak-positivity are recalled in section \ref{background material}. Observe that the assumptions of the theorem are satisfied for example if $X \arrow S$ is an abelian scheme of maximal variation. Our proof of Theorem \ref{main result} easily extends to the logarithmic setting.\\

Theorem \ref{main result} (and its extension to the logarithmic setting) is not new: it can easily be deduced from general results on variations of Hodge structures \cite{GriffithsIII, Zuo_neg, BKT, Bruni_Crelle, Bruni-Cado, Bruni_semipositivity}, using that Calabi-Yau manifolds satisfy the infinitesimal Torelli property (cf. section \ref{families of Calabi-Yau varieties}). The goal of this note is to present a proof which does not use any differential geometric notions, such as metrics and curvature computations, and relies ultimately on the nonabelian Hodge theory in positive characteristic developed by Ogus and Vologodsky \cite{Ogus-Vologodsky}.\\

We give also an algebraic proof of the following higher-dimensional version of the Arakelov inequality \cite{Arakelov, Faltings-Arakelov, Peters-Rigidity, Kim-ABC, Jost-Zuo, peters2000arakelovtype, Viehweg-Arakelov, Zuo_neg}.
\begin{theorem}[Arakelov inequality]\label{Arakelov inequality}
Same assumptions as Theorem \ref{main result}. If $n$ is the dimension of $S$ and $d$ is the relative dimension of $f$, then 
\[ [f_\ast(\omega_{X/S})] \leq \frac{1}{n+d} \cdot \dbinom{n+d}{n+1} \cdot [\omega_S] . \]
\end{theorem}

In the theorem, $[f_\ast(\omega_{X/S})]$ and $[\omega_S]$ are the divisors classes associated to the line bundles $f_\ast(\omega_{X/S})$ and $\omega_S$ respectively, and for any two $\bQ$-divisors classes $A$ and $B$ the notation $A \leq B$ means that any line bundle representing (up to a positive rational number) the difference $B-A$ is weakly-positive.\\

In a previous paper, we proved with analytic techniques \cite[Theorem 1.6]{brunebarbe2020increasing} that the preceding inequality holds with $\frac{d}{2}$ instead of $\frac{1}{n+d} \cdot \dbinom{n+d}{n+1}$, but we were unable to prove it algebraically when the base is not a curve. Note that, in particular, the constant should only depend on $d$ but not on $n$.

\subsection{Examples}

Let $\cA_g$ be the moduli stack of principally polarized complex abelian varieties of dimension $g$, and $\pi : \bA_g \arrow \cA_g$ be the universal family. Since the Hodge line bundle $\pi_ \ast \omega_{\bA_g / \cA_g}$ is known to be ample on (the coarse moduli space of) $\cA_g$ (see \cite{Moret-Bailly85} for an algebraic proof), it follows from Theorem \ref{main result} that any smooth proper connected complex variety that admits a generically finite morphism to $\cA_g$ is of general type and has a cotangent bundle which is weakly-positive and big in the sense of Hartshorne. Thanks to the Torelli Theorem, this applies in particular to any smooth proper complex variety that admits a generically finite morphism to the moduli stack $\cM_g$ of smooth projective curve of genus $g > 1$.\\

Similarly, for any polynomial $P \in \bQ[T]$, let $\cM_P$ be the stack that classifies pairs $(X, \cL)$ consisting  in a smooth projective complex variety $X$ with a trivial canonical bundle and an ample line bundle $\cL$ on $X$ whose Hilbert polynomial equals $P$, up to the equivalence relation $(X, \cL) \equiv (X^\prime, \cL^\prime)$ if there exists an isomorphism $\mu : X \arrow X^\prime$ such that $\mu^\ast \cL^\prime$ is numerically equivalent to $\cL$. Then $\cM_P$ is a separated Deligne-Mumford stack of finite type over $\bC$ and its coarse moduli space is quasi-projective \cite{Viehweg_book}. If $\pi : \cX \arrow \cM_P$ denote the universal family, then the Hodge line bundle $\pi_\ast \omega_{\cX / \cM_P}$ is known to be ample on (the coarse moduli space of) $\cM_P$, cf. \cite[Theorem 7.8]{BBT} (however, besides the case of abelian varieties and K3 surfaces \cite{Maulik}, it seems that there is no algebraic proof of this fact). Therefore, any smooth proper connected complex variety that admits a generically finite morphism to the stack $\cM_P$ satisfies the conclusions of Theorem \ref{main result}.

\subsection{Semi-positivity from Higgs bundles of geometric origin}

Let $S$ be a smooth complex scheme. Let $\HIG(S)$ denote the Abelian category of Higgs bundles $(\cE, \theta)$ on $S$, where $\cE$ is a quasicoherent $\cO_S$-modules and $\theta : \cE \arrow \cE \otimes_{\cO_S} \Omega_S^1$ is a $\cO_S$-linear morphism such that $\theta \wedge \theta = 0$. For any smooth proper morphism $f : X \arrow S$ and any integer $k$, we write $\rH^k_{Dol}(X/S)$ for the locally-free coherent $\cO_S$-module $\oplus_{p + q = k} \mathrm{R}^q f_\ast (\Omega_{X/S}^{p})$ endowed with the Gauss-Manin Higgs field. We say that a Higgs bundle on $S$ is of geometric origin if it belongs to the full subcategory of $\HIG(S)$ which contains the direct summands of the $\rH^k_{Dol}(X/S)$ for any smooth proper $S$-scheme $X$, and any extension of these factors.\\

The true mathematical content of this note consists in fact in explaining how one can reduce the proofs of Theorem \ref{main result} and Theorem \ref{Arakelov inequality} to the following result, using arguments from complex algebraic geometry.

\begin{theorem}[Beilinson-Deligne, Simpson]\label{semistability of Gauss-Manin Higgs bundles}
For any smooth proper morphism $f : X \arrow S$ with $S$ a smooth projective complex curve, and any integer $k$, the Higgs bundle $\rH^k_{Dol}(X/S)$ is semistable.
\end{theorem}

This result was first proved by Beilinson-Deligne (unpublished) and later, in a greater generality, by Simpson using transcendental tools. An algebraic proof was then given by Ogus and Vologodsky, as a consequence of their nonabelian Hodge theory in positive characteristic, cf. \cite[Theorem 4.21]{Ogus-Vologodsky}.\\

A key ingredient of our proof of Theorem \ref{main result} is given by the following technical result.

\begin{theorem}\label{semipositivityI}
Le $S$ be a smooth proper complex variety and $(\cE , \theta)$ be a Higgs bundle on $S$ of geometric origin. If $(\cF, \theta) \subset (\cE, \theta)$ is a Higgs subsheaf, then the line bundle $\det \cF^\vee$ is weakly positive. If moreover $\cF$ is a locally split $\cO_S$-submodule of $\cE$, then the line bundle $\det \cF^\vee$ is nef.
\end{theorem}

Theorem \ref{semipositivityI} can be proved with analytic methods using the special curvature properties of the Hodge metric. We offer an alternative proof which is algebraic and arguably simpler, with Theorem \ref{semistability of Gauss-Manin Higgs bundles} as the main input. As a direct corollary, we obtain an algebraic proof of the following special case of a result due to Zuo \cite{Zuo_neg} and the author \cite{Bruni_Crelle, Bruni_semipositivity}.

\begin{cor}\label{semipositivityII}
Le $S$ be a smooth proper complex variety $(\cE , \theta)$ be a Higgs bundle on $S$ of geometric origin. If $\cF \subset \cE$ is a $\cO_S$-submodule of $\cE$ contained in the kernel of $\theta$, then its dual $\cF^\vee$ is weakly positive. If moreover $\cF$ is a locally split $\cO_S$-submodule of $\cE$, then $ \cF^\vee$ is nef.
\end{cor}

As a direct application of Corollary \ref{semipositivityII}, we obtain yet another proof of the following well-known result of Griffiths:
\begin{theorem}
Let $A \arrow S$ be an abelian scheme over a smooth proper complex variety $S$. Then the vector bundle $\Omega^1_{A/S}$ is nef.
\end{theorem}

\begin{rem}
Note that thanks to Lefschetz principle all the preceding results are true over any field of characteristic zero (in fact all the proofs in this paper work in this generality). However, they are no longer true in characteristic $p > 2$, since Moret-Bailly has constructed non-isotrivial families of abelian surfaces $f: X \arrow \mathbb{P}^1$ such that $f_* \omega_{X/{\mathbb{P}^1}} = \mathcal{O}_{\bP^1}(p) \bigoplus \mathcal{O}_{\bP^1}(-1) $, cf. \cite{Moret-Bailly81}.
\end{rem}


\section{Positivity of torsion-free sheaves}\label{background material}
We recall for the reader convenience the definitions and properties of the different notions of positivity for torsion-free sheaves on smooth proper varieties that are used in this paper.\\

We start with the notions of weak positivity and bigness introduced by Viehweg, cf. \cite{Viehweg83} and \cite[ p.59-67]{Viehweg_book}. Recall that a coherent sheaf $\cF$ on a scheme $X$ is globally generated at a point $x \in X$ if the natural map $\HH^0(X, \cF) \otimes_{\bC} \cO_X \arrow \cF$ is surjective at $x$.

\begin{defn}[Viehweg]
Let $X$ be a smooth projective variety and $\cF$ a torsion-free sheaf on $X$. Let $i : V  \hookrightarrow X$ denotes the inclusion of the biggest open subset on which $\cF$ is locally free. 
\begin{enumerate}
\item We say that $\cF$ is weakly positive over the dense open subset $U \subset V$ if for every ample invertible sheaf $\cH$ on $X$ and every positive integer $\alpha  > 0$ there exists an integer $\beta > 0$ such that $\widehat{S}^{\alpha \cdot \beta} \cF \otimes_{ \cO_X} {\cH}^{\beta}$ is globally generated over $U$.
\item We say that $\cF$ is Viehweg-big over the dense open subset $U \subset V$ if for any line bundle $\cH$ there exists $\gamma > 0$ such that $\widehat{S}^{\gamma} \cF \otimes \cH^{-1}$ is weakly positive over $U$.
\end{enumerate}
(Here the notation $\widehat{S}^{k} \cF $ stands for the reflexive hull of the sheaf $S^{k} \cF $, i.e. $\widehat{S}^{k} \cF = i_* ( S^{k} i^*\cF)$.) We say that $\cF$ is weakly positive (resp. Viehweg-big) if there exists a dense open subset $U \subset V$ such that $\cF$ is weakly positive (resp. Viehweg-big) over $U$. 
\end{defn}

\begin{rem}
If $X$ is a smooth projective variety and $\cF$ is a locally free $\cO_X$-module, then $\cF$ is nef if and only if it is weakly positive over $X$, and $\cF$ is ample if and only if it is Viehweg-big over $X$.
\end{rem}

\begin{lem}[Viehweg]\label{Viehweg} Let $\cF$ and $\cG$ be torsion-free sheaves on a smooth projective variety $X$.

\begin{enumerate}

\item If $\cF \arrow \cG$ is a morphism, surjective over $U$, and if $\cF$ is weakly positive (resp. Viehweg-big) over $U$, then $\cG$ is weakly positive (resp. Viehweg-big) over $U$.

\item Let $f : Y \arrow X$ be a morphism between two smooth projective varieties. If $\cF$ is weakly positive over $U \subset X$ and $f^{-1}(U)$ is dense in $Y$, then $f^* \cF /(f^* \cF)_{tors}$ is weakly positive over $f^{-1}(U)$.

\item If $f:X^\prime \arrow X $ is a  birational  morphism between two smooth projective varieties which is an isomorphism over $U \subset X$, then  $\cF$ is weakly positive (resp. Viehweg-big) over $U$ if and only if $f^* \cF /(f^* \cF)_{tors}$ is weakly positive (resp. Viehweg-big) over $f^{-1}(U)$.

\item If $\cF$ is weakly positive and $\cH$ is a Viehweg-big line bundle, then $\cF \otimes \cH$ is Viehweg-big.

\end{enumerate}
\end{lem}

If $\cF$ is a torsion-free sheaf on a smooth proper (but non-necessarily projective) variety $X$, we say that $\cF$ is weakly positive (resp. Viehweg-big) if for any birational morphism $f:X^\prime \arrow X $ from a smooth projective variety, the torsion-free sheaf $f^* \cF /(f^* \cF)_{tors}$ is weakly positive (resp. Viehweg-big). This definition extends the definition for torsion-free sheaves on projective varieties thanks to the preceding lemma. \\

We introduce now a weaker notion of bigness. If $\cE$ is a vector bundle on a smooth proper variety $X$, we let $\pi : \bP(\cE) := \Proj_{\cO_X}(\Sym \, \cE) \arrow X$ be the projective bundle of one-dimensional quotients of $\cE$ and $\cO_{\cE}(1) $ be the tautological line bundle which fits in an exact sequence $ \pi^\ast \cE \arrow  \cO_{\cE}(1) \arrow 0$.

\begin{defn}[cf. {\cite[Definition 2.6]{Bruni_AENS}} and {\cite[Lemma 2.5]{Bruni_AENS}}]\label{Hartshorne-big}
A vector bundle $\cE$ on a smooth proper variety $X$ is called big (in the sense of Hartshorne) if it satisfies one of the following equivalent conditions:
\begin{enumerate}
\item The line bundle $ \cO_{\cE}(1)$ is Viehweg-big.
\item For some (resp. any) Viehweg-big line bundle $\cH$, there exists an injective map $ 0 \arrow \cH \arrow S^k \cE$ for some $k >0$. 
\item For some (resp. any) Viehweg-big torsion-free sheaf $\cF$, there exists a non-zero map $\cF \arrow S^k \cE$ for some $k >0$. 
\end{enumerate}
\end{defn}
A Viehweg-big vector bundle is big in the sense of Hartshorne, but the converse is not true (consider for example the rank $2$ vector bundle $ \cO(1) \oplus  \cO(-1)$ on $\bP^1$). Note however that the two notions coincide for line bundles.

\section{Families of varieties with trivial canonical bundle}\label{families of Calabi-Yau varieties}

According to Griffiths, the following observation is due to Andreotti.

\begin{lem}\label{Infinitesimal Torelli}
Let $X$ be a smooth proper complex variety of dimension $d$. If the canonical bundle $\omega_X$ of $X$ is trivial, then the canonical $\bC$-linear map 
\[ \HH^1(X , T_X) \otimes_{\bC} \HH^0(X, \omega_X) \arrow \HH^1(X, \Omega_X^{d-1}) \]
is an isomorphism.
\end{lem}

In particular, any smooth proper variety with a trivial canonical bundle satisfies the infinitesimal Torelli property.\\

The following result relates the positivity of the Hodge line bundle of a family of projective Calabi-Yau varieties to the injectivity of the Kodaira-Spencer map. 

\begin{prop}\label{bigness of the Hodge bundle implies KS injective}
Let $S$ be a smooth proper connected complex variety. Let $f: X \arrow S$ be a smooth projective morphism with connected fibres. Assume that the relative canonical bundle $\omega_{X/S}$ is relatively trivial. If the Hodge line bundle $f_\ast(\omega_{X/S})$ is big, then the Kodaira-Spencer map $T_S \arrow R^1 T_{X/S}$
is an injective morphism of $\cO_S$-modules.
\end{prop}
Using analytic techniques, it is in fact possible  to prove that for any family of smooth proper complex varieties with trivial canonical bundle, the bigness of the Hodge line bundle is equivalent to the injectivity of the Kodaira-Spencer map. However, we do not know an algebraic proof of this fact.

\begin{proof}
Let $\cL$ be a line bundle on $X$ which is relatively ample with respect to $f$. Since the Hilbert polynomial $P$ of $(X_s, \cL_{|X_s})$ is independent of the point $s \in S$, we get a morphism $S \arrow \cM_P$ with values in the corresponding moduli stack of polarized pairs modulo numerical equivalence. It is a separated Deligne-Mumford stack of finite type over $\bC$ \cite{Viehweg_book}. We denote by $M_P$ the associated coarse moduli space. The Hodge line bundle $f_\ast(\omega_{X/S})$ is the pull-back of a line bundle on $\cM_P$, hence some positive power of the Hodge line bundle is the pullback of a line bundle on $M_P$. Since the Hodge line bundle on $S$ is big by assumption, the morphism $S \arrow M_P$ is generically finite. But the Kodaira-Spencer map of the universal family on $\cM_P$ is an injective morphism of $\cO_{\cM_P}$-modules. It follows that the Kodaira-Spencer map $T_S \arrow R^1 T_{X/S}$ is a generically injective morphism of $\cO_S$-modules, and even injective everywhere since $T_S$ is locally-free hence torsion-free.

\end{proof}
\section{Semi-positivity from Higgs bundles}

\begin{defn}\label{def semistable}
A Higgs bundle $(\cE, \theta)$ on a smooth proper curve is said semistable if it has no Higgs subbundle $(\cF, \theta)$ with $\deg(\cF) > \deg(\cE)$. More generally, a Higgs bundle $(\cE, \theta)$ on a smooth proper variety $S$ is said semistable if for every smooth proper curve $C$ equipped with a morphism $f : C \arrow S$ the Higgs bundle $f^\ast (\cE,\theta)$ is semistable.
\end{defn}
Note that by the very definition, the pull-back of a semistable Higgs bundle is semistable (in particular, this definition of semistability is stronger than the usual definition of semistability with respect to an ample divisor class). The following result is the main input of this note (see below for the proof).

\begin{theorem}\label{geometric origin are semistable}
Any Higgs bundle of geometric origin on a smooth proper complex variety is semistable and its first rational Chern class is zero. 
\end{theorem}

One can prove that all the higher rational Chern classes are zero too, but we won't need it. In view of the preceding result, Theorem \ref{semipositivityI} and Corollary \ref{semipositivityII} are consequences of the following general results.

\begin{prop}\label{semipositivityI for semistable}
Le $S$ be a smooth proper complex variety. Let $(\cE , \theta)$ be a semistable Higgs bundle on $S$ with a vanishing first rational Chern class. If $(\cF, \theta) \subset (\cE, \theta)$ is a Higgs subsheaf, then the line bundle $\det \cF^\vee$ is weakly positive. If moreover $\cF$ is a locally split $\cO_S$-submodule of $\cE$, then the line bundle $\det \cF^\vee$ is nef.
\end{prop}

\begin{cor}\label{semipositivityII for semistable}
Le $S$ be a smooth proper complex variety. Let $(\cE , \theta)$ be a semistable Higgs bundle on $S$ with a vanishing first rational Chern class. If $\cF \subset \cE$ is a $\cO_S$-submodule of $\cE$ contained in the kernel of $\theta$, then its dual $\cF^\vee$ is weakly positive. If moreover $\cF$ is a locally split $\cO_S$-submodule of $\cE$, then $ \cF^\vee$ is nef.
\end{cor}

\begin{proof}[Proof of Proposition \ref{semipositivityI for semistable}]
Let $( \cE , \theta)$ be a semistable Higgs bundle with a vanishing first rational Chern class and $(\cF, \theta) \subset (\cE, \theta)$ be a Higgs subsheaf. By definition, for any smooth projective curve $C$ equipped with a morphism $f : C \arrow S$, the pulled-back Higgs bundle $f^\ast (\cE,\theta)$ is semistable. If $\cF$ is a locally split $\cO_S$-submodule of $\cE$, then $f^\ast (\cF, \theta)$ is a Higgs subsheaf of $f^\ast (\cE, \theta)$ and by semistability $f^\ast \det \cF = \det f^\ast \cF \leq \det f^\ast \cE = 0$. Hence $\det \cF^\vee$ is nef. In general, a standard argument using resolution of singularities implies the existence of a birational morphism $\nu : S^\prime \arrow  S$ from a smooth projective variety $S^\prime$ such that $\nu^\ast \cF$ has a morphism to a locally split $\cO_{S^\prime}$-submodule $\cG$ of $\nu^\ast \cE$, which is generically an isomorphism. The Higgs field of $\nu^\ast (\cE, \theta)$ stabilizes $\cG$, hence the line bundle $\det \cG^\vee$ is nef. Since there is a morphism $\det \cG^\vee \arrow \det (\nu^\ast \cF)^\vee $ which is generically an isomorphism, it follows that $\det (\nu^\ast \cF)^\vee = \nu^\ast (\det \cF^\vee) $ and $\det \cF^\vee$ are weakly positive thanks to Lemma \ref{Viehweg}.
\end{proof}

\begin{proof}[Proof of Corollary \ref{semipositivityII for semistable}]
As in the proof of Proposition \ref{semipositivityI for semistable}, it is sufficient to consider the case where $\cF$ is a locally split $\cO_S$-submodule of $\cE$ which is contained in the kernel of the Higgs field $\theta$. Let $\pi : \bP(\cF^\vee) := \Proj_{\cO_S}(\Sym \, \cF^\vee) \arrow S$ be the projective bundle of one-dimensional subspace of $\cF$ and $\cO_{\cF^\vee}(-1) $ be the tautological line bundle which fits in an exact sequence $0 \arrow  \cO_{\cF^\vee}(-1) \arrow  \pi^\ast \cF$. Then $\pi^\ast \cE$ equipped with the induced Higgs field $\pi^\ast \theta$ is a semistable Higgs bundle and $\pi^\ast \cF$ is contained in the kernel of $\pi^\ast \theta$. A fortiori, $\cO_{\cF^\vee}(-1)$ is contained in the kernel of $\pi^\ast \theta$, so that $(\cO_{\cF^\vee}(-1), 0)$ can be viewed as a Higgs subsheaf of $(\pi^\ast \cE, \pi^\ast \theta)$. Since $\cO_{\cF^\vee}(-1)$ is a locally split $\cO_{\bP(\cF^\vee)}$-submodule of $\pi^\ast \cE$, we get from Proposition \ref{semipositivityI for semistable} that $\cO_{\cF^\vee}(-1)^\vee = \cO_{\cF^\vee}(1)$ is a nef line bundle, so that $\cF^\vee$ is nef.
\end{proof}

\begin{rem}
The conclusions of Proposition \ref{semipositivityI for semistable} and Corollary \ref{semipositivityII for semistable} hold more generally for any Higgs bundle with zero rational Chern numbers $c_1(E) = c_2(E) = 0$ which is semistable with respect to an ample divisor. The proof is similar, once we know that these properties are preserved by pull-back (see \cite[Theorem 12]{Langer-inventionnes} for an algebraic proof of this fact).
 \end{rem}

\begin{proof}[Proof of Theorem \ref{geometric origin are semistable}]
Let $(\cE, \theta)$ be a Higgs bundle of geometric origin on a smooth proper variety $S$. Thanks to the base-change formula, the pull-back of a Higgs bundle of geometric origin is also of geometric origin. On the other hand, a locally free coherent $\cO_S$-module has a vanishing first rational Chern class if it has degree zero in restriction to any curve. Therefore it is sufficient to consider the case where the base $S$ is a smooth proper curve.\\

Given an extension $0 \arrow (\cE_1, \theta_1) \arrow (\cE, \theta) \arrow (\cE_2, \theta_2)\arrow 0$ of Higgs bundles on a smooth proper curve, if both $(\cE_1, \theta_1)$ and  $(\cE_2, \theta_2)$ are semistable of degree zero, then $(\cE, \theta)$ is semistable of degre zero, and the converse holds if the extension splits. Therefore, one is reduced to consider the case where $(\cE, \theta) = \rH^i_{Dol}(X/S)$ for a smooth proper morphism $f : X \arrow S$ and an integer $i$. The underlying $\cO_S$-module has degree zero, since it is isomorphic to the graded object associated to the relative De Rham cohomology $\rH^i_{DR}(X/S)$ endowed with its Hodge filtration, and recalling that the relative De Rham cohomology is naturally equipped with a flat connection (the Gauss-Manin connection) so that all its rational Chern classes are zero. Finally, the semistability of the Higgs bundle $ \rH^i_{Dol}(X/S)$ can either be seen as an easy consequence of the computation of the curvature of the Hodge metric  by Griffiths \cite{GriffithsIII}, or it can also be proved algebraically with the help of the nonabelian Hodge theory in positive characteristic of Ogus and Vologodsky, cf. \cite[Theorem 4.21]{Ogus-Vologodsky}.
\end{proof}

\section{Hyperbolicity from Higgs bundles}

\subsection{Weak-positivity of the cotangent bundle}
Let $S$ be a smooth proper complex variety, and let $(\cE, \theta)$ be a Higgs bundle on $S$. The Higgs field $\theta : \cE \arrow \Omega^1_S \otimes_{\cO_S} \cE$ is equivalently seen as a $\cO_S$-linear map of $\cO_S$-modules $\phi : T_S \arrow  \End(\cE)$.

\begin{prop}\label{Weak-positivity of the cotangent bundle}
Assume that the Higgs bundle $(\cE, \theta)$ is semistable with a vanishing first rational Chern class and that the associated $\cO_S$-linear map $\phi : T_S \arrow  \End(\cE)$ is injective. Then the cotangent bundle $\Omega_S^1$ of $S$ is weakly-positive.
\end{prop}

In view of the following lemma, the proposition is a direct consequence of Corollary \ref{semipositivityII for semistable} applied to the Higgs bundle $\End (\cE, \theta)$ and its $\cO_S$-submodule $T_S$.

\begin{lem}
If $\Theta : \End(\cE) \arrow  \Omega_S^1 \otimes_{\cO_S} \End(\cE)$ denotes the induced Higgs field on $\End(\cE)$, then the composition of $\phi : T_S \arrow  \End(\cE)$ with $\Theta$ is zero.
\end{lem}
\begin{proof}[Proof of the Lemma]
The induced Higgs field $\Theta$ on $\End(\cE)$ satisfies:
$$ (\Theta_s(\Psi))(v) = \theta_s(\Psi(v)) - \Psi(\theta_s(v)) $$
for $\Psi$ a section of $\End(\cE)$, $v$ a section of $\cE$ and $s$ a section of $T_S$.
Letting $\Psi = \theta_t$ for a section $t$ of $T_S$, we obtain 
$$ (\Theta_s(\theta_t))(v) = \theta_s(\theta_t(v)) - \theta_t(\theta_s(v)) = 0 $$
since $\theta \wedge \theta = 0$.
\end{proof}

\subsection{Bigness of the cotangent bundle}

\begin{prop}\label{Bigness of the cotangent bundle}
Let $S$ be a smooth proper complex variety, and let $(\cE, \theta)$ be a Higgs bundle of geometric origin on $S$. Assume that $\cE$ contains a big line bundle (or more generally a $\cO_S$-submodule whose determinant is big). Then the cotangent bundle of $S$ is big in the sense of Hartshorne (hence a fortiori $S$ is of general type).
\end{prop}
\begin{rem}
The same proof works more generally for any semistable Higgs bundle with a vanishing first rational Chern classe and whose Higgs field is nilpotent.
\end{rem}
\begin{proof}
Let $(\cE, \theta)$ be a Higgs bundle of geometric origin on $S$ and assume that $\cE$ contains a $\cO_S$-submodule $\cF$ of rank $r$ such that the line bundle $\det \cF$ is big. Since the $r$-th tensor power of  exterior of $(\cE, \theta)$ is also of geometric origin and contains the big line bundle $\det \cF$, one can assume from the beginning that $\cF$ is a line bundle.\\

By iterating the Higgs field $\theta$, we get for every positive integer $k$ a morphism of $\cO_S$-modules $\phi_k : \Sym^k T_S  \otimes \cF \arrow \cE$. Since the Higgs field is nilpotent, the $\phi_k$ are zero if $k \gg 1$. On the other hand, the morphism $\phi_1$ is nontrivial. Indeed otherwise $\cF$ would be contained in the kernel of the Higgs field, hence its dual $\cF^\vee$ would be weakly-positive thanks to Corollary \ref{semipositivityII}, in contradiction with the assumption that $\cF$ is big. If $k$ is the biggest (positive) integer for which $\phi_k$ is nonzero, then its image $\cN_1$ is a $\cO_S$-submodule of $\cE$ contained in the kernel of the Higgs field. Therefore, its dual $\cN_1^\vee$ is weakly-positive thanks to Corollary \ref{semipositivityII}. Since there is a non-zero morphism $  \cN_1^\vee \otimes \cF \arrow \Sym^k \Omega^1_S$ and the $\cO_S$-submodule $\cN_1^\vee \otimes \cF$ is big in the sense of Viehweg, it follows that $\Omega^1_S$ is big in the sense of Hartshorne, cf. Definition \ref{Hartshorne-big}.
\end{proof}

\section{Proofs of Theorem \ref{main result} and Theorem \ref{Arakelov inequality}}

Let $S$ be a smooth proper connected complex variety of dimension $n$. Let $f: X \arrow S$ be a smooth projective morphism with connected fibres of relative dimension $d$ such that the relative canonical bundle $\omega_{X/S}$ is relatively trivial. Consider the Higgs bundle $(\cE, \theta) := \rH^d_{Dol}(X/S)$. Therefore $\cE = \bigoplus_{k = 0}^d \cE^k$ with $\cE^k := \mathrm{R}^{d-k} f_\ast (\Omega_{X/S}^{k})$, so that in particular $\cE^d$ is the Hodge line bundle $f_\ast(\omega_{X/S})$, and $\theta (\cE^k) \subset \cE^{k-1} \otimes_{\cO_S} \Omega_S^1$ for every integer $k$.\\

From now on, we assume that the Hodge line bundle $\cE^d = f_\ast(\omega_{X/S})$ is big. In particular, it follows from Lemma \ref{Infinitesimal Torelli} and Proposition \ref{bigness of the Hodge bundle implies KS injective} that the morphism of $\cO_S$-modules $T_S  \otimes_{\cO_S} \cE^d \arrow \cE^{d-1}$ is injective. A fortiori, the $\cO_S$-linear map $\phi : T_S \arrow  \End(\cE)$ associated to $\theta$ is injective, hence the cotangent bundle $\Omega_S^1$ of $S$ is weakly-positive thanks to Proposition \ref{Weak-positivity of the cotangent bundle}.\\

Since $\cE$ contains the big line bundle $\cE^d$, we know from Proposition \ref{Bigness of the cotangent bundle} that $\Omega_S^1$ is big in the sense of Hartshorne. A fortiori, $S$ is of general type thanks to a general result of Campana-P\u{a}un \cite{Campana-Paun}. But in our situation this can also be seen in an elementary way as follows. As in the proof of Proposition \ref{Bigness of the cotangent bundle}, by iterating the Higgs field, we get for every positive integer $k$ a morphism of $\cO_S$-modules $\phi_k : \Sym^k T_S  \otimes_{\cO_S} \cE^d \arrow \cE^{d-k}$. If we take for $k$ the biggest integer for which $\phi_k$ is nonzero, then its image $\cN_1$ is a $\cO_S$-submodule of $\cE$ contained in the kernel of the Higgs field (note that $k \geq 1$ since $\phi_1$ is injective). Therefore, its dual $\cN_1^\vee$ is weakly-positive thanks to Corollary \ref{semipositivityII}. Since there is a non-zero morphism $  \cN_1^\vee \otimes_{\cO_S} \cE^d \arrow \Sym^k \Omega^1_S$ and the $\cO_S$-submodule $\cN_1^\vee \otimes_{\cO_S} \cE^d$ is big in the sense of Viehweg, we recover that $\Omega^1_S$ is big in the sense of Hartshorne. Consider the exact sequence of coherent $\cO_S$-modules:
\[ 0 \arrow \cN_2 \arrow \Sym^k T_S \arrow \cN_1 \otimes_{\cO_S} (\cE^d)^{\vee} \arrow 0. \]

Taking determinants we get that 
\[\det( \Sym^k \Omega^1_S) = (\cE^d)^{\otimes \rk \cN_1} \otimes_{\cO_S} (\det \cN_1 \otimes \det \cN_2)^{\vee}.\] 

Note that $\cN_2^\vee$ is weakly-positive, since $\Omega_S^1$ is weakly-positive and there is a generically surjective $\cO_S$-linear map $\Sym^k \Omega^1_S \arrow \cN_2^\vee$. A fortiori, the line bundle $(\det \cN_2)^{\vee}$ is weakly-positive. We have also seen that $\cN_1^\vee$ is weakly-positive, therefore $(\det \cN_1)^{\vee}$ is weakly-positive too. Finally, using that $\det( \Sym^k \Omega^1_S) = (\omega_S)^{\otimes C_k}$ with $ C_k = \dbinom{n + k -1}{n}$, we obtain that
\[ (\omega_S)^{\otimes C_k} \geq (\cE^d)^{\otimes \rk \cN_1} .\] 

(Given two line bundles $A$ and $B$ on $S$, the notation $A \geq B$ means that the line bundle $A \otimes_{\cO_S} B^{\vee}$ is weakly-positive.) Using that $\rk \cN_1 \geq 1$ and $\dbinom{n + k -1}{n} \leq \dbinom{n + d -1}{n} $, we have thus proved the following Arakelov inequality
\[ \cE^d \leq  (\omega_S)^{\otimes \dbinom{n + d -1}{n}}. \]
In particular $\omega_S$ is a big line bundle. This finishes the proof of Theorem \ref{main result}.\\

To obtain the slightly better constant in the Arakelov inequality of Theorem \ref{Arakelov inequality}, we can argue as follows. For every integer $k \geq 1$, let $\cG^{d- k}$ denote the image of $\phi_k : \Sym^k T_S \otimes_{\cO_S} \cE^d \arrow \cE^{d- k}$. If we set $\cG^d := \cE^d$ and $\cG := \bigoplus_{k= 0}^d \cG^k$, then $(\cG, \theta)$ is by definition a Higgs subsheaf of $(\cE, \theta)$ (and it is the smallest Higgs subsheaf of $(\cE, \theta)$ containing $\cE^d$). Let as before $r$ denote the biggest integer $k$ such that $\phi_k$  is not zero, so that $\cG^{d-k} = 0$ if $k > r$. Define the $\cK^i$'s so that the following sequence
\[ 0 \arrow \cK^{d-k} \arrow \Sym^k T_S  \arrow \cG^{d- k} \otimes_{\cO_S} (\cE^d)^{\vee} \arrow 0 \]
is exact for every non-negative integer $k$. We know thanks to Theorem \ref{main result} that the cotangent bundle of $S$ is weakly positive. Therefore the dual of $\det  \cK^{d-k}$ is weakly-positive for every positive $k$, and we get
\[ \det \Sym^k T_S  \leq \det \left(\cG^{d- k} \otimes_{\cO_S} (\cE^d)^{\vee} \right) = \det(\cG^{d- k}) \otimes_{\cO_S} \left((\cE^d)^{\vee} \right)^{\otimes \rk(\cG^{d-k})}. \]
Since that this holds also trivially for $k = 0$, we get that
\[ \otimes_{k=0}^r \left( \det ( \Sym^k T_S ) \otimes_{\cO_S} (\cE^d)^{ \otimes \rk(\cG^{d-k})} \right) \leq \otimes_{k=0}^r \det(\cG^{d- k}) = \det \cG. \]
Recall that $\det( \Sym^k \Omega^1_S) = (\omega_S)^{\otimes C_k}$ with $ C_k = \dbinom{n + k -1}{n}$, and that 
\[\sum_{k=0}^r  \dbinom{n+k-1}{n} = \dbinom{n+r}{n +1}.\]
Since $\det \cG \leq 0$ thanks to Proposition \ref{semipositivityI for semistable}, it follows that
\[ (\cE^d)^{\otimes \rk(\cG)} \leq (\omega_S) ^{\otimes \dbinom{n+r}{n+1}}. \]

Note that $\rk(\cG^d) = 1$, $\rk(\cG^{d-1}) = n$ (since $\phi_1$ is injective) and $\rk(\cG^{d-k}) \geq 1$ for any $2 \leq k \leq r$, so that $\rk(\cG) \geq 1 + n + (r-1) = n+ r$. Since the quantity $\frac{1}{n+r} \cdot \dbinom{n+r}{n+1}$ grows with $r$, we finally get that 
\[ (\cE^d)^{d+1} \leq (\omega_S) ^{\otimes \dbinom{n+d}{n+1}}. \]

\section{A few questions}

\subsection{}
For any smooth proper morphism $f : X \arrow S$ with $S$ a smooth projective complex curve, and any integer $k$, the Higgs bundle $\rH^k_{Dol}(X/S)$ is semistable. In fact, using harmonic metrics, Simpson proved that they are even polystable. Find an algebraic proof of this fact.

\subsection{} Find an algebraic proof that the Hodge line bundle is ample on any moduli stack of polarized Calabi-Yau varieties (up to numerical equivalence).

\subsection{} Find an algebraic proof that the Arakelov inequality in Theorem \ref{Arakelov inequality} holds with $C = \frac{d}{2}$.

\bibliographystyle{alpha}
\bibliography{biblio}

\end{document}